\documentclass[10pt, a4paper, reqno, english]{amsart}

\usepackage{hyperref}
\usepackage{amssymb,mathrsfs,bbm,dsfont}
\usepackage[all]{xy}
\usepackage{enumitem}
\usepackage{tikz}
\usetikzlibrary{matrix,arrows}
\usepackage{comment}

\newcommand\ZZ{\mathbb{Z}}

\newcommand\CC{\mathbb{C}}
\newcommand\RR{\mathbb{R}}

\newcommand\QQ{\mathbb{Q}}

\newcommand\PP{\mathbb{P}} 
\newcommand{\OO}{\mathcal{O}} 
\DeclareMathOperator{\Cl}{Cl} 
\DeclareMathOperator{\meas}{meas} 

\newcommand{\idealfull}{\mathcal F_m}
\newcommand{\idealfulltriv}{\mathcal F_{m,0}}
\newcommand{\idealres}{\mathcal R_m} 

\newtheorem{theorem}{Theorem}
\newtheorem{lemma}[theorem]{Lemma}
\newtheorem{prop}[theorem]{Proposition}
\newtheorem{cor}[theorem]{Corollary}
\theoremstyle{definition}

\newtheorem{proposition}[theorem]{Proposition}

\definecolor{darkred}{rgb}{0.5,0,0}
  \definecolor{darkgreen}{rgb}{0,0.5,0}
  \definecolor{darkblue}{rgb}{0,0,0.7}
  \hypersetup{linkcolor=darkblue,filecolor=darkgreen,urlcolor=darkred,citecolor=darkblue}
\newcommand{\marta}[1]{{\color{darkred} \sf $\spadesuit$Marta:[#1]}}

\begin{document}

\setcounter{tocdepth}{1}

\title{On $m$-fullness in number fields}

\author{Michiel Horikx}
\author{Marta Pieropan}

\address{Heinrich Heine University D\"usseldorf, Mathematical Institute, Universit\"atsstr. 1, 40225 D\"usseldorf, Germany \emph{and} 
Utrecht University, Mathematical Institute, Budapestlaan 6, 3584 CD Utrecht, The Netherlands}
\email{marta.pieropan@hhu.de}
\email{m.m.l.horikx@gmail.com}

\date{\today}



\begin{abstract}  
We prove asymptotic formulas for the number of $m$-full ideals of bounded norm in a given ideal class and for the number of $m$-full algebraic integers of bounded Weil height.
\end{abstract}

\maketitle

\section{Introduction}
Fix a positive integer $m$. Recall that an integer $\alpha$ is called $m$-full if $p^m\mid \alpha$ for all prime divisors $p$ of  $\alpha$.
Estimates for $m$-full integers of bounded size are given in \cite{ES34, BG58}.
More precisely, the number of $m$-full positive integers of size up to $x$ is
\begin{equation}\label{eq:m-full integers}
x^{\frac 1m} \prod_{p} \left(1+\sum_{i=m+1}^{2m-1}p^{-\frac im}\right) +O\left(x^{\frac 1{m+1}}\right).
\end{equation}

In this paper we study $m$-fullness of ideals and of algebraic integers in number fields. 
In analogy with the case of integers, we say that an ideal $I$ of the ring of integers $\OO_K$ of a number field $K$ is $m$-full if $I\subseteq \mathfrak p^m$ for all prime ideals $\mathfrak p$ of $\OO_K$ that contain $I$. 
As first result we give an asymptotic formula for the number of $m$-full ideals of bounded norm in prescribed ideal classes.

\begin{theorem}\label{thm}
Let $m\geq 2$ be an integer.
Let $K$ be a number field. Let $d_K$ be the degree of $K$ over $\QQ$, and let $\Cl(K)$ be the class group of $K$. 
For every ideal class $c\in\Cl(K)$ and for all $x\in\RR_{>0}$, let $N_{m,c}(x)$ be the number of $m$-full ideals of $\OO_K$ of ideal class $c$ and of norm up to $x$. Then
$$
N_{m,c}(x) = x^{\frac 1m}\frac{\rho_K\#\Cl(K)[m]}{h_K}\sum_{\substack{I\in \idealres\\ [I]\in c+m\Cl(K)}}N(I)^{-\frac 1m} + O(x^{\nu}(\log(x+1))^\mu),
$$
where $\rho_K$ is the residue of the Dedekind zeta function of $K$ at $s=1$, $\#\Cl(K)[m]$ is the cardinality of the $m$-torsion subgroup of $\Cl(K)$, $h_K$ is the class number of $K$, $\idealres$ is the set of ideals $I$ such that $I\subseteq \mathfrak p^{m+1}$ and $I\nsubseteq \mathfrak p^{2m}$ for all prime ideals $\mathfrak p$ that contain $I$,
$$
\nu = \begin{cases}\frac 1{m+1} & \text{ if } d_K\leq m+1, \\ \frac{d_K-1}{md_K} & \text{ otherwise,} \end{cases}
\qquad \text{ and } \qquad
\mu=\begin{cases} 1 & \text{ if } d_K=m+1,\\ 0 & \text{ otherwise.} \end{cases}
$$
\end{theorem}

Theorem \ref{thm} is proven by comparing the Dirichlet series associated to $N_{m,c}(x)$ to the Dirichlet series that detects ideals that are $m$-th powers and lie in the ideal class $c$.
If $K=\mathbb Q$, Theorem \ref{thm} recovers \eqref{eq:m-full integers}. 
A direct generalisation of \eqref{eq:m-full integers} to ideals is given by the following corollary.

\begin{cor}\label{cor}
Let $K$ be a number field, and let $m$ be a positive integer.
For all $x\in\RR_{>0}$, let $N_{m}(x)$ be the number of $m$-full ideals of $\OO_K$ of norm up to $x$. Then
$$
N_{m}(x) = x^{\frac 1m} \rho_K \prod_{\mathfrak p}\left(1+\sum_{i=m+1}^{2m-1}N(\mathfrak p)^{-\frac im}\right) + O(x^{\nu}\log(x+1)^\mu),
$$
with $\rho_K$, $h_K$, $\nu$ and $\mu$ as in Theorem \ref{thm}.
\end{cor}

We say that an algebraic integer $\alpha\in \OO_K$ is $m$-full if so is the ideal $\alpha\OO_K$.
We recall that the 
height of an algebraic integer $\alpha\in \OO_K$ is $H(\alpha)=\prod_{v\in\Omega_{\infty}}\max\{1,|\alpha|_v\}$, where the product runs over the set $\Omega_{\infty}$ of infinite places of $K$ and $|\cdot |_v$ is normalised as in Section \ref{sec:notation}.
As second result we prove an asymptotic formula for the number of $m$-full algebraic integers of bounded height.

\begin{theorem}\label{thm:algebraic}
Let $K$ be a number field, and let $m$ be a positive integer. Let $\mu_K$ be the group of roots of unity in $K$, and $R_K$ the regulator of $K$. Let  $r_1$ and $r_2$ be the numbers of real and complex places of $K$, respectively. For all $x\in\RR_{>0}$, let $N'_m(x)$ be the number of $m$-full elements of $\OO_K$ of height at most $x$. Then 
$$
N'_m(x)=
\frac{\#\mu_K\rho_K C_K C_{K,m}}{h_K}  x^{\frac 1m}(\log x)^{r_1+r_2-1} + E
$$
where 
$$
C_K=\begin{cases}
1 & \text{ if } r_1+r_2=1,\\
((r_1+r_2-1)!R_K)^{-1} & \text{ if }r_1+r_2\geq 2,
\end{cases}
$$
$$
C_{K,m} = \begin{cases}
1 & \text{ if } m=1,\\
\#\Cl(K)[m]\sum_{\substack{I\in \idealres\\ [I]\in m\Cl(K)}}N(I)^{-\frac 1m} & \text{ if } m\geq 2,
\end{cases}
$$ 
$$
E= \begin{cases} 
O(x^{\nu}(\log(x+1))^\mu) & \text{ if } r_1+r_2=1, \\
O(x^{\frac 1m}(\log x)^{r_1+r_2-2}\log\log x) & \text{ if } r_1+r_2\geq 2,
\end{cases}
$$
and $\rho_K$, $h_K$, $\Cl(K)$, $\idealres$, $\nu$ and $\mu$ are as in Theorem \ref{thm}.
\end{theorem}

The number $N'_m(x)$ has a geometric interpretation as the cardinality of the set of Campana $\OO_K$-points of Weil height at most $x$ on the dlt Campana pair $(\PP^1_K, (1-\frac 1m)\{0\}+\{\infty\})$ in the sense of \cite{PSTVA21}.

The proof of Theorem \ref{thm:algebraic} combines Theorem \ref{thm} with estimates for the number of algebraic integers of fixed norm and bounded height that we provide in Proposition \ref{prop:integers of given norm}. 

The special case $m=1$ of Theorem \ref{thm:algebraic} recovers a known asymptotic formula for algebraic integers of bounded height. For this, a claim  without explicit leading constant can be found in \cite[Theorem 3.5.2]{Lan83}. We refer to \cite{Bar14, Wid16, GG17} for the history of this special case and its generalisations. 

\subsection{Notation}
\label{sec:notation}
Let $m$ be a positive integer.
Let $K$ be a number field, 
and $\mathcal O_K$ its ring of integers. 
Let $\mathcal I$ be the set of nonzero ideals of $\mathcal O_K$, 
and $\mathcal P$ the set of nonzero prime ideals of $\mathcal O_K$. 
For $I\in\mathcal I$, we denote by $N(I)$ the ideal norm of $I$, 
and by $[I]$ the ideal class of $I$.
Let $\Cl(K)$ be the class group of $K$.
Let $\Cl(K)[m]$ be the $m$-torsion subgroup of $\Cl(K)$ and $m\Cl(K)$ the image of the multiplication by $m$ endomorphism of $\Cl(K)$.
The residue of the Dedekind zeta function $\zeta_K(s)$ of $K$ at $s=1$ is 
$$
\rho_K=\frac{2^{r_1}(2\pi)^{r_2}R_Kh_K}{\#\mu_K|\Delta_K|^{\frac 12}},
$$
where $r_1$ and $r_2$ are the number of real and complex places of $K$, respectively, $R_K$ is the regulator of $K$, $h_K$ is the class number of $K$, $\mu_K$ is the group of roots of unity in $K$, and $|\Delta_K|$ is the absolute value of the discriminant of $K$.
Let $d_K=[K:\QQ]$.
Let $\Omega_{\infty}$ be the set of infinite places of $K$.
For every place $v$ of $K$ lying over a place $w$ of $\mathbb Q$, let $|\cdot|_v=|N_{K_v/\QQ_w}(\cdot)|_w$, where $K_v$ and $\QQ_w$ are the completions of $K$ and $\QQ$ at $v$ and $w$, respectively, $N_{K_v/\QQ_w}:K_v\to\QQ_w$ is the norm, and $|\cdot|_w$ is the usual $p$-adic or real absolute value on $\QQ$.

We denote by $\prod_{\mathfrak p}$ a product over all $\mathfrak p\in\mathcal P$, and we denote by $\sum_{N(I) \dots}$ a sum over all ideals $I\in\mathcal I$ such that $N(I)\dots$ holds.

For real valued functions $f$ and $g$, we use the symbols $f\ll g$ and $f=O(g)$ to indicate  that there is a positive constant $C$ such that $|f(x)|\leq C g(x)$ for all $x$. The notation $\Re(s)\gg 1$ indicates $\Re(s)\geq C$ for some positive constant $C$ independent of $s$. 

We denote by $\RR_{\geq 0}$ the set of nonnegative real numbers, and for every integer $n$ we denote by $\ZZ_{\geq n}$ the set of integers that are at least $n$.

\section{Counting tool}
\label{sec:counting_tool}
We say that a function $a:\mathcal I\to\CC$ defines a convergent Dirichlet series if the series
$$
\sum_{I\in\mathcal I}a(I)N(I)^{-s}
$$
converges for $s\in\CC$ with $\Re(s)\gg 1$.

\begin{proposition}\label{prop:bateman-grosswald}
Let $a,b,c:\mathcal I\to\CC$ be functions that define convergent Dirichlet series and such that 
\begin{equation}
\label{eq:product}
\left(\sum_{I\in\mathcal I}\frac{a(I)}{N(I)^{s}}\right)\left(\sum_{I\in\mathcal I}\frac{b(I)}{N(I)^{s}}\right)=\sum_{I\in\mathcal I}\frac{c(I)}{N(I)^{s}}
\end{equation}
for all $s\in\CC$ with $\Re(s)\gg 1$.

Assume that 
$$
\sum_{N(I)\leq x}a(I)=\alpha_0x^{\lambda_0}+\dots + \alpha_r x^{\lambda_r} + O(x^{\lambda}\log^{\mu}(x+1)) \quad \text{ and } \quad \sum_{N(I)\leq x}|b(I)| = O(x^\nu)
$$
with $\lambda, \mu, \nu \in\RR_{\geq 0}$, $\alpha_0,\dots,\alpha_r, \lambda_0,\dots,\lambda_r\in\CC$.
Then 
$$
\sum_{N(I)\leq x}c(I) = \gamma_0x^{\lambda_0}+\dots + \gamma_r x^{\lambda_r} +O(x^{\max(\lambda,\nu)}(\log(x+1))^{\mu'}),
$$
where for $i\in\{0,\dots,r\}$,
$$
\gamma_i = \begin{cases}
\alpha_i \sum_{I\in\mathcal I} b(I)N(I)^{-\lambda_i} & \text{ if } \Re(\lambda_i)>\nu,\\
0 & \text{ otherwise,}
\end{cases}
$$
and
$$
\qquad \mu'=\begin{cases}
\mu & \text{ if }\lambda>\nu,\\
\mu+1 & \text{ if }\lambda = \nu,\\
1 & \text{ if } \lambda <\nu \text{ and } \Re(\lambda_i) = \nu \text{ for some } i\in\{0,\dots,r\},\\
0 & \text{ otherwise.} 
\end{cases}
$$
\end{proposition}
\begin{proof}

For all positive integers $n$, let 
$$a_n=\sum_{N(I)=n}a(I), \quad b_n=\sum_{N(I)=n}b(I), \quad c_n=\sum_{N(I)=n}c(I).$$ 
By \eqref{eq:product} we have $\left( \sum_{n\geq 1} a_n n^{-s}\right)\left(\sum_{n\geq 1}b_n n^{-s} \right) = \sum_{n\geq 1} c_n n^{-s}$ for $s\in\CC$ with $\Re(s)\gg 1$. Moreover,
$$\sum_{n\leq x}a_n = \sum_{N(I)\leq x}a(I), \quad \sum_{n\leq x}|b_n|\leq \sum_{N(I)\leq x}|b(I)|, \quad \sum_{n\leq x}c_n = \sum_{N(I)\leq x}c(I).
$$
Then we conclude by \cite[Lemma 1]{BG58}.
\end{proof}

\section{$m$-full ideals in a fixed ideal class}

For every ideal $I\in\mathcal I$ and every prime ideal $\mathfrak p\in\mathcal P$, we define $v_\mathfrak p(I)\in\ZZ_{\geq 0}$ by 
$I=\prod_{\mathfrak p}\mathfrak p^{v_{\mathfrak p}(I)}$.
Then
$$
\idealres = \{I\in\mathcal I: v_{\mathfrak p}(I)\in\{0
\}\cup\{m+1,\dots,2m-1\} \ \forall \mathfrak p\in\mathcal P\}.
$$
The set of $m$-full ideals and the set of ideals that are $m$-th powers are
\begin{gather*}
\idealfull = \{I\in\mathcal I: v_{\mathfrak p}(I)\in\{0\}\cup \ZZ_{\geq m} \ \forall \mathfrak p\in\mathcal P\},\\
\mathcal I_m=\{I\in\mathcal I: m\mid v_{\mathfrak p}(I) \ \forall \mathfrak p\in\mathcal P\}, 
\end{gather*}
respectively.
By the unique factorization of ideals in Dedekind domains, there are bijections
\begin{gather}
 \mathcal I \to \mathcal I_m, \qquad J\mapsto J^m, \label{eq:bij_power}
\\ \idealres \times \mathcal I_m \to \idealfull, \qquad (I,J)\mapsto IJ. \label{eq:bij_full}
\end{gather}

For every subset $S\subseteq \mathcal I$ and every ideal class $c\in\Cl(K)$, we denote by $\mathds 1_{S}$ the characteristic function of $S$ and by $\mathds 1_{S,c}$ the characteristic function of the set $\{I\in S: [I]=c\}$.
Then 
\begin{equation}\label{eq:number_full}
N_{m,c}(x)=\sum_{N(I)\leq x}\mathds 1_{\idealfull,c}(I).
\end{equation}

Before proving Theorem \ref{thm}, we recall some estimates for $m$-th powers and  residues in a given ideal class.
\begin{proposition}\label{prop:bounds}
Let $m$ be a positive integer.
Let $c\in\Cl(K)$. For $x\in\RR_{\geq 0}$,
\begin{gather*}
\sum_{N(I)\leq x}\mathds 1_{\mathcal I_m, c}(I)=\frac{\rho_K\#\Cl(K)[m]}{h_K} x^{\frac 1m}  +O\left(x^{\frac{d_K-1}{md_K}}\right),\\
\sum_{N(I)\leq x}\mathds 1_{\idealres,c}(I) = O(x^{\frac 1{m+1}}).
\end{gather*}
\end{proposition}
\begin{proof}
We prove first the asymptotic formula. By the bijection \eqref{eq:bij_power}, writing $I=J^m$ with $J\in\mathcal I$ gives
\begin{gather*}
\sum_{N(I)\leq x}\mathds 1_{\mathcal I_m,c}(I) 
= \sum_{\substack{c'\in\Cl(K)\\ mc'=c}}\sum_{\substack{N(J)\leq x^{\frac 1m}\\ [J]=c'}}1.
\end{gather*}
By \cite[Theorem VI.3.3]{Lan94}
the inner sum is 
\begin{equation}\label{eq:ideals of bounded norm}
\frac{\rho_K}{h_K} x^{\frac 1m} +O\left(x^{\frac {d_K-1}{md_K}}\right).
\end{equation} 
Since the leading term is independent of the choice of the ideal class $c'$ we get the asymptotic formula in the statement.

For the upper bound, write $\prod_{i=1}^{m-1}\zeta_K((m+i)s)=\sum_{I\in\mathcal I}a(I)N(I)^{-s}$. 
For all prime ideals $\mathfrak p\in\mathcal P$ and all $e\in\ZZ_{\geq 0}$ we have 
$$a(\mathfrak p^e)
=\#\{(e_1,\dots,e_{m-1})\in\ZZ_{\geq 0}^{m-1}:e_1(m+1)+\dots+e_{m-1}(2m-1)=e\}. 
$$
Thus $a(\mathfrak p^e)\geq \mathds 1_{\idealres}(\mathfrak p^{e})$, and 
$$
a(I) = \prod_{\mathfrak p}a(\mathfrak p^{v_{\mathfrak p(I)}})\geq \prod_{\mathfrak p}\mathds 1_{\idealres}(\mathfrak p^{v_{\mathfrak p(I)}})=\mathds 1_{\idealres}(I)
$$ 
for all $I\in\mathcal I$.
Then
$$
\sum_{N(I)\leq x}\mathds 1_{\idealres,c}(I) \leq\sum_{N(I)\leq x}\mathds 1_{\idealres}(I) \leq\sum_{N(I)\leq x}a(I).
$$
Since $\prod_{i=1}^{m-1}\zeta_K((m+i)s)$ is a Dirichlet series with nonnegative coefficients, holomorphic for $\Re(s)>\frac 1{m+1}$ with a simple pole at $s=\frac 1{m+1}$ and no other poles on the line $\Re(s)=\frac 1{m+1}$, a Tauberian theorem \cite[Theorem A]{PTBZ25}
gives the desired bound.
\end{proof}

\begin{proof}[Proof of Theorem \ref{thm}]
By the bijection \eqref{eq:bij_full}, every $I\in\idealfull$ can be written uniquely as $I=\widetilde I J$ with $\widetilde I\in \idealres$ and $J\in\mathcal I_m$. Let $\widetilde c=[\widetilde I]$. Then $[J]=[I]-\widetilde c\in m\Cl(K)$. Thus
$$
\sum_{I\in\mathcal I}\frac{\mathds 1_{\idealfull,c}(I)}{N(I)^{s}}=\sum_{\widetilde c\in c+m\Cl(K)}\left(\sum_{\widetilde I\in\mathcal I}\frac{\mathds 1_{\idealres,\widetilde c}(\widetilde I)}{N(\widetilde I)^{s}} \right)\left(\sum_{J\in\mathcal I}\frac{\mathds 1_{\mathcal I_m,c-\widetilde c}(J)}{N(J)^{s}}\right).
$$
An application of Proposition \ref{prop:bateman-grosswald} with $a(I)=\mathds 1_{\mathcal I_m,c-\widetilde c}(I)$ and $b(I)=\mathds 1_{\idealres,\widetilde c}(I)$ and with the estimates in Proposition \ref{prop:bounds} gives
$$
N_{m,c}(x) = x^{\frac 1m}\frac{\rho_K\#\Cl(K)[m]}{h_K}\sum_{\substack{I\in \idealres\\ [I]\in c+m\Cl(K)}}N(I)^{-\frac 1m} + O(x^{\nu}(\log(x+1))^\mu),
$$
where 
$$
\nu = \max\left\{\frac 1{m+1},\frac{d_K-1}{md_K}\right\}
\qquad \text{ and } \qquad
\mu=\begin{cases} 1 & \text{ if } \frac 1{m+1}=\frac{d_K-1}{md_K},\\ 0 & \text{ otherwise.} \end{cases}
$$
We conclude as $\frac 1{m+1}\geq \frac{d_K-1}{md_K}$ if and only if $d_K\leq m+1$.
\end{proof}
\begin{proof}[Proof of Corollary \ref{cor}]
By Theorem \ref{thm} the leading constant is $$
\frac{\rho_K\#\Cl(K)[m]}{h_K}\sum_{I\in \idealres}\sum_{c\in [I]+m\Cl(K)}N(I)^{-\frac 1m} = \rho_K\sum_{I\in \idealres}N(I)^{-\frac 1m}.$$
We conclude by computing the Euler product, as $\mathds 1_{\idealres}$ is a multiplicative function.
\end{proof}

\section{Algebraic integers of given norm}

In order to prove an asymptotic formula for the number of algebraic integers of given norm and bounded height, we need to compute the volume of a certain polytope.

\begin{lemma}\label{lemma:volume polytope}
Let $n$ be a positive integer. For real numbers $0\leq c\leq b$,  let 
$$P_{n,b,c}=\{\underline a\in\RR^n: c-b\leq \sum_{i\in S}a_i\leq b \ \forall S\subseteq \{1,\dots,n\},  S\neq\emptyset\}.$$
Then there is a homogeneous polynomial $Q_n\in\QQ[u,v]$ of degree $n$ such that $$\meas_n(P_{n,b,c})=Q_n(b,c),$$ where $\meas_n$ is the Lebesgue measure on $\RR^n$. In particular,
$$
Q_n(u,v)=\sum_{k=0}^n {n \choose k} \frac{u^k(u-v)^{n-k}}{k!(n-k)!}, \qquad Q_n(1,1)=\frac 1{n!}, \qquad Q_n(b,c)=O(b^n),
$$
and the implicit constant in the upper bound is independent of $c$.

For every coordinate subspace $V\subseteq \RR^n$, let $p_V:\RR^n\to V$ be the orthogonal projection onto $V$.
Then $p_V(P_{n,b,c}) = P_{\dim V,b,c}$.
\end{lemma}
\begin{proof}
For every subset $T\subseteq\{1,\dots,n\}$, let $T^c=\{1,\dots,n\}\smallsetminus T$, and let
$$
R_T=\{\underline a\in\RR^n: a_i\geq 0 \ \forall i\in T, a_j\leq 0 \ \forall j\in T^c\}
$$
be the quadrant associated to $T$. Then $\meas_n(P_{n,b,c})=\sum_{T\subseteq\{1,\dots,n\}}\meas_n(P_{b,c}\cap R_T)$.
Fix $T\subseteq\{1,\dots,n\}$. 
We use the convention that a sum over the empty set is zero.
Then
$$
P_{n,b,c}\cap R_T = \left\{\underline a\in R_T: c-b\leq \sum_{i\in T^c}a_i, \ \sum_{i\in T}a_i \leq b\right\},
$$
as 
$$
 \sum_{i\in T^c}a_i \leq \sum_{i\in S\cap T^c}a_i \leq \sum_{i\in S} a_i \leq \sum_{i\in S\cap T}a_i \leq \sum_{i\in T}a_i
$$
for all nonempty $S\subseteq \{1,\dots,n\}$ and all $\underline a\in R_T$.
Thus 
$
\meas_n(P_{n,b,c}\cap R_T)=\meas_{\#T}(b\Delta^{\#T})\meas_{n-\#T}((b-c)\Delta^{n-\#T}),
$
where for every nonnegative integer $k$, 
$$\Delta^k=\left\{\underline a\in\RR^k: \sum_{i=1}^k a_i\leq 1, \ a_i\geq 0 \ \forall i\in\{1,\dots,k\}\right\}$$ is the $k$-dimensional simplex, and $\meas_k(\Delta^k)=\frac 1{k!}$. Thus
$$
\meas_n(P_{n,b,c})=\sum_{k=0}^n {n \choose k} \frac{b^k(b-c)^{n-k}}{k!(n-k)!}.
$$

Let $V\subseteq \RR^n$ be a coordinate subspace, and let $U\subseteq\{1,\dots,n\}$ such that $V=\{\underline a\in\RR^n: a_i=0\ \forall i\in\{1,\dots,n\}\smallsetminus U\}$. We observe that for every $T\subseteq\{1,\dots,n\}$ we have
$$
p_V(P_{n,b,c}\cap R_T)=\left\{\underline a\in V\cap R_T: c-b\leq \sum_{i\in U\cap T^c}a_i, \ \sum_{i\in U\cap T} a_i\leq b\right\}.
$$
Thus
$
p_V(P_{n,b,c}) = \bigcup_{T\subseteq \{1,\dots,n\}}p_V(P_{n,b,c}\cap R_T) = P_{\dim V,b,c}.
$
\end{proof}

The following proposition gives an asymptotic formula for the number of algebraic integers in $K$ of given norm and bounded height. The proof uses Dirichlet's unit theorem to reduce to a lattice point counting problem involving the polytope in the previous lemma.

\begin{prop}\label{prop:integers of given norm}
Let $I$ be a nonzero principal ideal of $\OO_K$. For all $x\in\RR_{>0}$, let $N''_{N(I)}(x)$ be the set of elements of $\OO_K$ of norm $N(I)$ and of height at most $x$.
Then for $x\geq N(I)$, we have $N''_{N(I)}(x)=\#\mu_K$ if $r_1+r_2=1$, and
$$
N''_{N(I)}(x) =  \frac{\#\mu_K Q_{r_1+r_2-1}(\log x,\log N(I))}{R_K} + O((\log x)^{r_1+r_2-2})
$$
if $r_1+r_2\geq 2$,
where 
$Q_{r_1+r_2-1}$ is the polynomial from Lemma \ref{lemma:volume polytope}, and the implicit constant in the error term is independent of $N(I)$.
\end{prop}
\begin{proof}
The number $N''_{N(I)}(x)$ is the cardinality of the set 
$$
A=\{\alpha\in\OO_K:N(\alpha)=N(I), H(\alpha)\leq x\}.
$$
Since 
$$
\ell: K^\times \to\RR^{\Omega_{\infty}}, \qquad \alpha\mapsto (\log|\alpha|_v)_{v\in\Omega_{\infty}},
$$
is a group homomorphism with kernel the group $\mu_K$ of roots of unity in $K$, 
$$
N''_{N(I)}(x)=\#\mu_K\#\ell(A).
$$
We recall that $\Lambda = \ell(\OO_K^\times)$ is a lattice in the hyperplane given by $\sum_{v\in\Omega_{\infty}}a_v=0$. Let $r=r_1+r_2=\#\Omega_{\infty}$, and let $\epsilon_1,\dots,\epsilon_{r-1}$ be independent fundamental units of $\OO_K$. Then $\Lambda$ is defined by the matrix $(\ell(\epsilon_i))_{1\leq i\leq r-1}$.
Let $\underline d=(d_v)_{v\in\Omega_{\infty}}$, with $d_v=1/d_K$ if $v$ is real and $d_v=2/d_K$ if $v$ is complex.
Let $\Lambda_{N(I)}=\Lambda + \log N(I)\underline d$.
Then
$$
\Lambda_{N(I)} =  \ell\left(\{\alpha\in\OO_K: N(\alpha)=N(I)\}\right),
$$
as units are the algebraic integers of norm $1$.
If $r=1$, then $\#\ell(A)=\#\Lambda_{N(I)}=1$. 
We assume that $r\geq 2$.
Since $H(\alpha)=\prod_{v\in\Omega_{\infty}}\max\{1,|\alpha|_v\}$, the condition $H(\alpha)\leq x$ is equivalent to $\prod_{v\in S}|\alpha|_v\leq x$ for all nonempty subsets $S\subseteq\Omega_{\infty}$. Thus
$$
\ell(A)=\left\{\underline a\in \Lambda_{N(I)}: \sum_{v\in S}a_v\leq \log x \ \forall S\subseteq \Omega_{\infty}, S\neq\emptyset\right\}.
$$
 Let $\tilde v\in\Omega_{\infty}$. 
Writing $a_{\tilde v}=\log N(I) - \sum_{v\in\Omega_{\infty}\smallsetminus\{\tilde v\}}a_v$, we get
$$
\ell(A)
=\left\{\underline a\in \Lambda_{N(I)}: \log N(I)-\log x\leq \sum_{v\in S}a_v\leq \log x \ \forall S\subseteq \Omega_{\infty}\smallsetminus\{\tilde v\}, S\neq\emptyset\right\}.
$$
The coordinate projection $\pi:\RR^{\Omega_{\infty}}\to\RR^{\Omega_{\infty}\smallsetminus\{\tilde v\}}$ induces an isomorphism from $\{\sum_{v\in\Omega_{\infty}}a_v=\log N(I)\}$ onto $\RR^{\Omega_{\infty}\smallsetminus\{\tilde v\}}$, so that 
$$
\#\ell(A)=\#\left(\pi(\Lambda_{N(I)})\cap P_{r-1,\log x,\log N(I)}\right),
$$
where $P_{r-1,\log x,\log N(I)}\subseteq\RR^{r-1}$ is the polytope in Lemma \ref{lemma:volume polytope}. 
We observe that $\pi(\Lambda_{N(I)})=\pi(\Lambda)+\log N(I) (d_v)_{v\in\Omega_{\infty}\smallsetminus\{\tilde v\}}$, and that $\pi(\Lambda)\subseteq \RR^{\Omega_{\infty}\smallsetminus\{\tilde v\}}$ is the full rank lattice  given by the matrix $B=(\log|\epsilon_i|_v)_{1\leq i\leq r-1, v\in\Omega_{\infty}\smallsetminus\{\tilde v\}}$. 
Translating by $-\log N(I) (d_v)_{v\in\Omega_{\infty}\smallsetminus\{\tilde v\}}$ yields a lattice point counting problem. Performing a change of variables by $B$ and using the fact that volume is invariant under translation, we can apply \cite{Dav51} to get 
\begin{gather*}
\#\ell(A) 
= \#(\pi(\Lambda)\cap (P_{r-1,\log x,\log N(I)}-\log N(I) (d_v)_{v\in\Omega_{\infty}\smallsetminus\{\tilde v\}}))
\\ = \frac {\meas_{r-1}(P_{r-1,\log x,\log N(I)})}{\det B} + O\left(\sum_V\meas_{\dim V}(p_V(P_{r-1,\log x,\log N(I)}))\right),
\end{gather*}
where the sum runs over all coordinate subspaces of $\RR^{\Omega_{\infty}\smallsetminus\{\tilde v\}}$. Since $R_K=\det B$ by definition, an application of Lemma \ref{lemma:volume polytope} gives
\begin{equation*}
\#\ell(A) = \frac{Q_{r-1}(\log x,\log N(I))}{R_K} + O((\log x)^{r-2}). 
\qedhere
\end{equation*}
\end{proof}

The special case $I=\OO_K$ of Proposition \ref{prop:integers of given norm} gives an asymptotic formula for the number of units in $\OO_K$ of bounded height. A claim without explicit leading constant can be found in \cite[Theorem 3.5.2]{Lan83}.

\section{$m$-full algebraic integers}

In this section we combine Theorem \ref{thm} and Proposition \ref{prop:integers of given norm} to prove Theorem \ref{thm:algebraic}.

\begin{proof}[Proof of Theorem \ref{thm:algebraic}]
Let $\idealfulltriv= \{I\in\idealfull: [I]=[\OO_K]\}$.
Since principal ideals of $\OO_K$ are equal if and only if they have the same norm,
$$N'_{m}(x)=1+ \sum_{I\in\idealfulltriv, N(I)\leq x} N''_{N(I)}(x).$$
If $r_1+r_2=1$ the result follows by Proposition \ref{prop:integers of given norm}, Theorem \ref{thm} for $m\geq 2$ and \eqref{eq:ideals of bounded norm} for $m=1$.
We assume that $r_1+r_2\geq 2$.
We apply Proposition \ref{prop:integers of given norm} and the bound $N_{m,[\OO_K]}(x)=O(x^{\frac 1m})$ from Theorem \ref{thm} for $m\geq 2$ and from \eqref{eq:ideals of bounded norm} for $m=1$, to get
$$
N'_m(x)=\frac{\#\mu_K}{R_K}\sum_{I\in\idealfulltriv, N(I)\leq x}Q(\log x,\log N(I)) +O(x^{\frac 1m}(\log x)^{r_1+r_2-2}),
$$
where $Q$ is the polynomial $Q_{r_1+r_2-1}$ from Lemma \ref{lemma:volume polytope}. 
Since $Q$ is a homogeneous polynomial, $$Q(\log x,\log N(I))=(\log x)^{r_1+r_2-1}Q\left(1,\frac{\log N(I)}{\log x}\right).$$ The Taylor series expansion of $Q(1,t)$ at $t=1$ gives 
\begin{equation*}\label{eq:Q simplified}
Q(\log x,\log N(I))=(\log x)^{r_1+r_2-1} Q(1,1) + O((\log x)^{r_1+r_2-2}(\log x-\log N(I)).
\end{equation*}
Hence,
\begin{gather*}
\sum_{I\in\idealfulltriv, N(I)\leq x}Q(\log x,\log N(I))
\\ =Q(1,1)(\log x)^{r_1+r_2-1} N_{m,[\OO_K]}(x) + O((\log x)^{r_1+r_2-2}(E_1 + E_2)),
\end{gather*}
where
$$
E_1 = \sum_{I\in\idealfulltriv, N(I)\leq \frac{x}{(\log x)^m}} (\log x-\log N(I)) \leq  N_{m,[\OO_K]}\left(\frac x{(\log x)^m}\right) \log x \ll x^{\frac 1m},
$$
and 
\begin{gather*}
E_2 = \sum_{I\in\idealfulltriv, \frac{x}{(\log x)^m}\leq N(I)\leq x} (\log x-\log N(I)) 
\\ \ll  N_{ m,[\OO_K]}(x) \log\log x \ll x^{\frac 1m}\log\log x,
\end{gather*}
as $\log x-\log N(I)\leq \log x-\log (\frac x{(\log x)^m})=m\log\log x$.
We conclude by applying Theorem \ref{thm} as $\nu = \max\left\{\frac 1{m+1},\frac{d_K-1}{md_K}\right\}<\frac 1m$.
\end{proof}

\subsection*{Acknowledgements}
This paper originated from the Master's thesis project of the first named author under the supervision of the second named author. 
The second named author thanks the organisers of the thematic program ``Universal Statistics in Number Theory'' for the hospitality of Centre de Recherches Math\'ematiques in Montreal where this paper was completed.
The second named author was partially supported by the Dutch Research Council (NWO) grants VI.Vidi.213.019 and OCENW.XL21.XL21.011.
For the purpose of open access, a CC BY public copyright license is applied to any Author Accepted Manuscript version arising from this submission.

\bibliographystyle{alpha}
\bibliography{m-fullness}

@article {Bar14,
    AUTHOR = {Barroero, F.},
     TITLE = {Counting algebraic integers of fixed degree and bounded
              height},
   JOURNAL = {Monatsh. Math.},
  FJOURNAL = {Monatshefte f\"ur Mathematik},
    VOLUME = {175},
      YEAR = {2014},
    NUMBER = {1},
     PAGES = {25--41},
      ISSN = {0026-9255,1436-5081},
   MRCLASS = {11G50 (11R04)},
  MRNUMBER = {3249885},
MRREVIEWER = {Christopher\ Frei},
       DOI = {10.1007/s00605-013-0599-6},
       URL = {https://doi-org.utrechtuniversity.idm.oclc.org/10.1007/s00605-013-0599-6},
}

@article {BG58,
    AUTHOR = {Bateman, P.~T. and Grosswald, E.},
     TITLE = {On a theorem of {E}rd\"os and {S}zekeres},
   JOURNAL = {Illinois J. Math.},
  FJOURNAL = {Illinois Journal of Mathematics},
    VOLUME = {2},
      YEAR = {1958},
     PAGES = {88--98},
      ISSN = {0019-2082},
   MRCLASS = {10.00},
  MRNUMBER = {0095804},
MRREVIEWER = {N. G. de Bruijn},
       URL = {http://projecteuclid.org/euclid.ijm/1255380836},
}

@article {Dav51,
    AUTHOR = {Davenport, H.},
     TITLE = {On a principle of {L}ipschitz},
   JOURNAL = {J. London Math. Soc.},
  FJOURNAL = {The Journal of the London Mathematical Society},
    VOLUME = {26},
      YEAR = {1951},
     PAGES = {179--183},
      ISSN = {0024-6107,1469-7750},
   MRCLASS = {10.0X},
  MRNUMBER = {43821},
MRREVIEWER = {W.\ H.\ Mills},
       DOI = {10.1112/jlms/s1-26.3.179},
       URL = {https://doi-org.utrechtuniversity.idm.oclc.org/10.1112/jlms/s1-26.3.179},
}

@article{ES34,
	author = {Erd\H{o}s, P. and Szekeres, G.},
	title = {{\"U}ber die {A}nzahl der {A}belschen {G}ruppen gegebener {O}rdnung und \"uber ein verwandtes zahlentheoretisches {P}roblem},
	journal = {Acta Univ. Szeged. Sect. Sci. Math.},
	volume = {7},
	year = {1934},
	number = {2},
	pages = {92 -- 102},
}

@article {GG17,
    AUTHOR = {Grizzard, R. and Gunther, J.},
     TITLE = {Slicing the stars: counting algebraic numbers, integers, and
              units by degree and height},
   JOURNAL = {Algebra Number Theory},
  FJOURNAL = {Algebra \& Number Theory},
    VOLUME = {11},
      YEAR = {2017},
    NUMBER = {6},
     PAGES = {1385--1436},
      ISSN = {1937-0652,1944-7833},
   MRCLASS = {11N45 (11G50 11H16 11P21 11R04 11R06)},
  MRNUMBER = {3687101},
MRREVIEWER = {Dzianis\ Kaliada},
       DOI = {10.2140/ant.2017.11.1385},
       URL = {https://doi-org.utrechtuniversity.idm.oclc.org/10.2140/ant.2017.11.1385},
}

@book {Lan83,
    AUTHOR = {Lang, S.},
     TITLE = {Fundamentals of {D}iophantine geometry},
 PUBLISHER = {Springer-Verlag, New York},
      YEAR = {1983},
     PAGES = {xviii+370},
      ISBN = {0-387-90837-4},
   MRCLASS = {11-02 (11Dxx 11Gxx 14G25)},
  MRNUMBER = {715605},
MRREVIEWER = {Gerd\ Faltings},
       DOI = {10.1007/978-1-4757-1810-2},
       URL = {https://doi-org.utrechtuniversity.idm.oclc.org/10.1007/978-1-4757-1810-2},
}

@book {Lan94,
    AUTHOR = {Lang, S.},
     TITLE = {Algebraic number theory},
    SERIES = {Graduate Texts in Mathematics},
    VOLUME = {110},
   EDITION = {Second},
 PUBLISHER = {Springer-Verlag, New York},
      YEAR = {1994},
     PAGES = {xiv+357},
      ISBN = {0-387-94225-4},
   MRCLASS = {11Rxx (11-01 11-02)},
  MRNUMBER = {1282723},
MRREVIEWER = {M.\ Ram\ Murty},
       DOI = {10.1007/978-1-4612-0853-2},
       URL = {https://doi-org.utrechtuniversity.idm.oclc.org/10.1007/978-1-4612-0853-2},
}

@article {PTBZ25,
    AUTHOR = {Pierce, L. B. and Turnage-Butterbaugh, C. L. and
              Zaman, A.},
     TITLE = {A guide to {T}auberian theorems for arithmetic applications},
   JOURNAL = {Essent. Number Theory},
  FJOURNAL = {Essential Number Theory},
    VOLUME = {5},
      YEAR = {2026},
    NUMBER = {2},
     PAGES = {271--374},
      ISSN = {2834-4626,2834-4634},
   MRCLASS = {11M45 (40E05)},
  MRNUMBER = {5124773},
       DOI = {10.2140/ent.2026.5.271},
       URL = {https://doi.org/10.2140/ent.2026.5.271},
}

@article {PSTVA21,
    AUTHOR = {Pieropan, M. and Smeets, A. and Tanimoto, S. and
              V\'{a}rilly-Alvarado, A.},
     TITLE = {Campana points of bounded height on vector group
              compactifications},
   JOURNAL = {Proc. Lond. Math. Soc. (3)},
  FJOURNAL = {Proceedings of the London Mathematical Society. Third Series},
    VOLUME = {123},
      YEAR = {2021},
    NUMBER = {1},
     PAGES = {57--101},
      ISSN = {0024-6115},
   MRCLASS = {11G50 (11G35 14G05 14G10)},
  MRNUMBER = {4307130},
       DOI = {10.1112/plms.12391},
       URL = {https://doi.org/10.1112/plms.12391},
}

@article {Wid16,
    AUTHOR = {Widmer, M.},
     TITLE = {Integral points of fixed degree and bounded height},
   JOURNAL = {Int. Math. Res. Not. IMRN},
  FJOURNAL = {International Mathematics Research Notices. IMRN},
      YEAR = {2016},
    NUMBER = {13},
     PAGES = {3906--3943},
      ISSN = {1073-7928,1687-0247},
   MRCLASS = {14G25 (11G05)},
  MRNUMBER = {3544624},
MRREVIEWER = {John\ L.\ Boxall},
       DOI = {10.1093/imrn/rnv268},
       URL = {https://doi-org.utrechtuniversity.idm.oclc.org/10.1093/imrn/rnv268},
}
\end{document}